\documentclass[11pt,twoside]{article}
\usepackage{amssymb,amsmath,amsthm,amsfonts,mathrsfs,hyperref}
\usepackage{times}
\usepackage{enumerate}
\usepackage{cite,titletoc}
\usepackage[toc,page,title,titletoc,header]{appendix}

\allowdisplaybreaks

\def\cB{\mathcal B}

\def\cI{\mathcal I}

\def\cK{\mathcal K}

\def\cM{\mathcal M}

\def\cO{\mathcal O}

\def\cQ{\mathcal Q}

\def\cW{\mathcal W}

\def\N{\mathop{\mathbb N\kern 0pt}\nolimits}
\def\Z{\mathop{\mathbb Z\kern 0pt}\nolimits}
\def\Q{\mathop{\mathbb Q\kern 0pt}\nolimits}
\def\R{\mathop{\mathbb R\kern 0pt}\nolimits}
\def\T{\mathop{\mathbb T\kern 0pt}\nolimits}
\def\SS{\mathop{\mathbb S\kern 0pt}\nolimits}

\def\ds{\displaystyle}

\def\al{\alpha}
\def\supp{\mathop{\rm supp}\nolimits}

\def\p{\partial}

\def\ve{\varepsilon}

\def\dive{\operatorname{div}}

\def\supp{\operatorname{supp}}
\def\ls{\lesssim}
\def\gt{\gtrsim}

\newcommand{\w}[1]{\langle {#1} \rangle}

\hypersetup{colorlinks=true,linkcolor=blue,citecolor=red,urlcolor=cyan}

\theoremstyle{plain}
\newtheorem{theorem}{Theorem}[section]

\newtheorem{lemma}[theorem]{Lemma}

\theoremstyle{definition}

\newtheorem{remark}{Remark}[section]

\numberwithin{equation}{section}

\title{Global existence for 2-D wave maps equation in exterior domains}

\author{Hou Fei$^{1,*}$ \quad Yin Huicheng$^{2,*}$ \quad  Yuan Meng$^{3,}$
    \footnote{Hou Fei (\texttt{fhou$@$nju.edu.cn}), Yin Huicheng (\texttt{huicheng$@$nju.edu.cn}, \texttt{05407$@$njnu.edu.cn})
    and Yuan Meng (\texttt{ym$@$cczu.edu.cn}) are supported by the NSFC (No.12331007, No.12101304). In addition, Hou Fei and Yin Huicheng are supported by the National key research and development program of China (No.2020YFA0713803, No.2024YFA1013301).}\\
    [12pt]{\small 1. School of Mathematics, Nanjing University, Nanjing, 210093, China}\\
    {\small 2. School of Mathematical Sciences and Mathematical Institute,}\\
    {\small Nanjing Normal University, Nanjing, 210023, China}\\
    {\small 3. School of Computer Science and Artificial Intelligence, Aliyun School of Big Data,}\\
    {\small School of Software, Changzhou University, Changzhou, 213164, China}}

\begin{document}

\date{}
\maketitle
\thispagestyle{empty}

\begin{abstract}
In the paper [H. Kubo, Global existence for exterior problems of semilinear wave equations with the null condition in 2D,
Evol. Equ. Control Theory 2 (2013), no. 2, 319-335], for the 2-D semilinear wave equation system
$(\p_t^2-\Delta)v^I=Q^I(\p_tv, \nabla_xv)$ ($1\le I\le M$) in the exterior domain with Dirichlet boundary condition,
it is shown that the small data smooth solution $v=(v^1, \cdot\cdot\cdot, v^M)$ exists globally when the
cubic nonlinearities $Q^I(\p_tv, \nabla_xv)=O(|\p_tv|^3+|\nabla_xv|^3)$ satisfy
the null condition.
We now focus on the global Dirichelt boundary value problem of 2-D wave maps equation
with the form $\ds\Box u^I=\sum_{J,K,L=1}^MC_{IJKL}u^JQ_0(u^K,u^L)$ $(1\le I\le M)$ and $\ds Q_0(f,g)=\p_tf\p_tg-\sum_{j=1}^2\p_jf\p_jg$
in exterior domain.
By establishing some crucial classes of pointwise spacetime decay estimates for the small data
solution $u=(u^1, \cdot\cdot\cdot, u^M)$ and its derivatives,
the global existence of $u$ is shown.

\vskip 0.2 true cm


\qquad \quad ghost weight, Hardy inequality

\vskip 0.2 true cm
\noindent
\end{abstract}

\vskip 0.2 true cm

\addtocontents{toc}{\protect\thispagestyle{empty}}
\tableofcontents

\section{Introduction}

Let ($\mathcal{M}$, $g$) be a $M$-dimensional compact Riemannian manifold without
boundary and the map $v=(v^1,\dots, v^M): \R\times(\R^2\setminus\cO)\rightarrow\cM$.
Assume that  $\eta$ is the standard Minkowski metric on $\R^{1+2}$ which
can be represented by the matrix $\mathrm{diag}(-1,1,1)$ in rectangular coordinates.
Define the functional $\ds L(v)=\int_{\R\times(\R^2\setminus\cO)}\w{\eta^{\alpha\beta}\p_{\alpha}v,\p_{\beta}v}_{g}dtdx$,
where $\cO\subset \R^2$ is a compact obstacle.
A 2-D wave map $u=(u^1,\dots, u^M)$ is the critical point of the functional $\ds L(v)$.
When the local coordinates on $\mathcal{M}$ are introduced, the equations of $u^I~(1\leq I\leq M)$ can be written as
\begin{equation}\label{Tao;wavemap}
\Box u^I=\sum_{J,K=1}^2\Gamma^I_{JK}(u)\big(\p_t u^J\p_t u^K-\sum_{L=1}^2\p_{x_L} u^J\p_{x_L} u^K\big),
\end{equation}
where $\Gamma^I_{JK}(u)$'s are the Christoffel symbols of the metric $g$. For more detailed derivation of the wave maps
equation \eqref{Tao;wavemap}, one can see Chapter 2 of \cite{Shatah98}. With respect to the small data solution
$u$, through taking the Taylor expansion on $\Gamma^I_{JK}(u)$ under Riemann normal coordinates, \eqref{Tao;wavemap} is essentially equivalent to
\begin{equation}\label{WaveMap}
\Box u^I=\sum_{J,K,L=1}^MC_{IJKL}u^JQ_0(u^K,u^L),
\end{equation}
where $\ds Q_0(f,g)=\p_tf\p_tg-\sum_{j=1}^2\p_jf\p_jg$ and $C_{IJKL}$ are constants.

Instead of \eqref{WaveMap}, we now study the following initial boundary value
problem (which is abbreviated as IBVP) of the more general nonlinear wave equation system
\begin{equation}\label{NLW}
\left\{
\begin{aligned}
&\Box u^I=F^I(u,\p u,\p^2u)=\sum_{J,K,L=1}^M\sum_{\cQ\in\{Q_0,Q_{\alpha\beta}\}}C_{IJKL}u^J\cQ(u^K,u^L),\quad (t,x)\in(0,+\infty)\times\cK,\\
&u(t,x)=0,\hspace{6.2cm} (t,x)\in(0,+\infty)\times\p\cK,\\
&(u,\p_tu)(0,x)=(\ve u_0,\ve u_1)(x),\qquad\qquad\qquad\qquad\quad x\in\cK,
\end{aligned}
\right.
\end{equation}
where $x=(x_1,x_2)$, $\p=(\p_0,\p_1,\p_2)=(\p_t, \p_{x_1}, \p_{x_2})$, $u=(u^1, \cdot\cdot\cdot, u^M)$, $Q_{\alpha\beta}(f,g)=\p_{\alpha}f\p_{\beta}g-\p_{\beta}f\p_{\alpha}g$ for $\alpha,\beta=0,1,2$, $\cK=\R^2\setminus\cO$,
the obstacle $\cO\subset \R^2$ contains the origin and is star-shaped, and the boundary $\p\cK=\p\cO$ is smooth.
In addition, $u_i=(u_i^1, \cdot\cdot\cdot, u_i^M)$ for $i=0,1$,
$(u_0,u_1)\in C^\infty(\cK)$ and $\supp(u_0,u_1)\subset\{x\in\cK:|x|\le M_0\}$ with
some fixed constant $M_0>1$. Obviously, \eqref{WaveMap} is a special case of the equation system in \eqref{NLW}.
Our main result is
\begin{theorem}\label{thm1}
Suppose that $(u_0,u_1)$ satisfies the compatibility conditions up to $(2N+1)-$order ($N\ge 42$) with respect to
the initial boundary values of \eqref{NLW}. Then there is an $\ve_0>0$ such that for $\ve\le\ve_0$ and
\begin{equation}\label{initial:data}
\begin{split}
&\|u_0\|_{H^{2N+1}(\cK)}+\|u_1\|_{H^{2N}(\cK)}\le1,
\end{split}
\end{equation}
problem \eqref{NLW} admits a global solution
$u\in\bigcap\limits_{j=0}^{2N+1}C^{j}([0,\infty), H^{2N+1-j}(\cK))$.
Moreover, there is a constant $C>0$ such that
\addtocounter{equation}{1}
\begin{align}
\sum_{|a|\le N}|\p Z^au|&\le C\ve\w{x}^{-1/2}\w{t-|x|}^{-1},\tag{\theequation a}\label{thm1:decay:a}\\
\sum_{|a|\le N}\sum_{i=1}^2|\bar\p_iZ^au|&\le C\ve\w{x}^{-1/2}\w{t+|x|}^{0.001-1},\tag{\theequation b}\label{thm1:decay:b}\\
\sum_{|a|\le2N-29}|Z^au|&\le C\ve\w{t+|x|}^{0.001-1/2},\tag{\theequation c}\label{thm1:decay:c}
\end{align}
where $Z=\{\p,\Omega\}$ with $\Omega=x_1\p_2-x_2\p_1$, $\w{x}=\sqrt{1+|x|^2}$,  $\bar\p_i=\frac{x_i}{|x|}\p_t+\p_i$ ($i=1,2$)
are the good derivatives (tangent to the outgoing light cone $|x|=t$).
In addition, the following time decay estimate of local energy  holds
\begin{equation}\label{thm1:decay:LE}
\sum_{|a|\le2N-27}\|\p^au\|_{L^2(\cK_R)}\le C_R\ve(1+t)^{-1},
\end{equation}
where $R>1$ is any fixed constant, $\cK_R=\cK\cap\{x: |x|\le R\}$ and  $C_R>0$ is a constant
depending on $R$.
\end{theorem}

\begin{remark}\label{Rem-1}
{\it Consider the following nonlinear wave equation system
\begin{equation}\label{rmk:NLW}
\Box u^I=\sum_{J,K,L=1}^M\sum_{\alpha,\beta=0}^2Q^{\alpha\beta}_{IJKL}u^J\p_{\alpha}u^K\p_{\beta}u^L,
\end{equation}
where the null condition satisfies
\begin{equation}\label{rmk:null}
\sum_{\alpha,\beta=0}^2Q^{\alpha\beta}_{IJKL}\xi_{\alpha}\xi_{\beta}\equiv0\quad\text{for any $(\xi_0,\xi_1,\xi_2)\in\{\pm1\}\times\SS$
and $I, J, K, L=1,\cdots, M$.}
\end{equation}
In terms of \cite{Klainerman}, one knows from \eqref{rmk:null} that there are some constants $C^0_{IJKL}$ and $C^{\alpha\beta}_{IJKL}$ such that
\begin{equation*}
\begin{split}\sum_{J,K,L=1}^M\sum_{\alpha,\beta=0}^2Q^{\alpha\beta}_{IJKL}u^J\p_{\alpha}u^K\p_{\beta}u^L
&=\sum_{J,K,L=1}^MC^0_{IJKL}u^JQ_0(u^K,u^L)+\sum_{J,K,L=1}^M\sum_{0\le\alpha<\beta\le2}C^{\alpha\beta}_{IJKL}u^JQ_{\alpha\beta}(u^K,u^L).
\end{split}
\end{equation*}
Therefore, problem \eqref{NLW} is equivalent to \eqref{rmk:NLW} with \eqref{rmk:null} and the corresponding initial boundary value conditions.
}
\end{remark}

\begin{remark}\label{Rem-2}
{\it Collecting the arguments in this paper and in \cite{HouYinYuan24},
Theorem \ref{thm1} can be extended into the following initial boundary value
problem of the fully nonlinear wave equation system
\begin{equation*}\label{GWE}
\begin{split}
&\Box u^I=\sum_{J,K=1}^M\sum_{|a|,|b|\le1}C_{IJK}^{ab}Q_0(\p^au^J,\p^bu^K)
+\sum_{J,K,L=1}^M\sum_{|a|,|b|\le1}\sum_{\cQ\in\{Q_0,Q_{\alpha\beta}\}}C_{IJKL}^{ab}u^J\cQ(\p^au^K,\p^bu^L)\\
&\quad +\sum_{J,K,L=1}^M\Big\{\sum_{\alpha,\beta,\gamma,\delta=0}^2Q_{IJKL}^{\alpha\beta\gamma\delta}\p^2_{\alpha\beta}u^J\p_{\gamma}u^K\p_{\delta}u^L
+\sum_{\alpha,\beta,\gamma=0}^2S_{IJKL}^{\alpha\beta\gamma}\p_{\alpha}u^J\p_{\beta}u^K\p_{\gamma}u^L\\
\end{split}
\end{equation*}
\begin{equation}\label{GWE}
\begin{split}
&\quad+\sum_{\alpha,\beta,\gamma,\delta,\mu,\nu=0}^2F_{IJKL}^{\alpha\beta\gamma\delta\mu\nu}\p^2_{\alpha\beta}u^J\p^2_{\gamma\delta}u^K\p^2_{\mu\nu}u^L
+\sum_{\alpha,\beta,\gamma,\delta,\mu=0}^2F_{IJKL}^{\alpha\beta\gamma\delta\mu}\p^2_{\alpha\beta}u^J\p^2_{\gamma\delta}u^K\p_{\mu}u^L\Big\},
\end{split}
\end{equation}
where $C_{IJK}^{ab}, C_{IJKL}^{ab}, Q_{IJKL}^{\alpha\beta\gamma\delta}, S_{IJKL}^{\alpha\beta\gamma}, F_{IJKL}^{\alpha\beta\gamma\delta\mu\nu}$ and $F_{IJKL}^{\alpha\beta\gamma\delta\mu}$ are
constants, and the related null conditions and symmetric conditions hold
as in problem (1.1) and Remark 1.2 of  \cite{HouYinYuan24}.}
\end{remark}

\begin{remark}\label{Rem-3}
{\it Note that the compatibility conditions of $(u_0,u_1)$ in Theorem \ref{thm1} are necessary
in order to find smooth solution $u$ of \eqref{NLW}.
For readers' convenience, we now give the illustrations on the compatibility conditions: Set $J_ku=\{\p_x^\al u:0\le|\al|\le k\}$
and $\p_t^ku(0,x)=F_k(J_ku_0,J_{k-1}u_1)$ ($0\le k\le 2N$), where $F_k$ depends on the nonlinear forms in \eqref{NLW},
$J_ku_0$ and $J_{k-1}u_1$. The compatibility conditions for problem \eqref{NLW} up to $(2N+1)-$order mean
that all $F_k$ vanish on $\p\cK$ for $0\le k\le 2N$.}
\end{remark}

\begin{remark}\label{Rem-4}
{\it For the Cauchy problem of the nonlinear wave equations
\begin{equation}\label{QWE:ivp}
\left\{
\begin{aligned}
&\Box u=Q(\p u, \p^2u),\hspace{2.6cm} (t,x)\in(0,+\infty)\times\R^n,\\
&(u,\p_tu)(0,x)=(\ve u_0,\ve u_1)(x),\qquad x\in\R^n,
\end{aligned}
\right.
\end{equation}
where $(u_0,u_1)(x)\in C_0^{\infty}(\Bbb R^n)$, $\ve>0$ is sufficiently small,
and the nonlinearity $Q(\p u, \p^2u)=O(|\p u|^2+|\p^2u|^2)$ is at least quadratic in $(\p u,\p^2u)$,
so far there have been extensive and systematical results on the global existence or blowup of smooth solutions.
For examples, when $n\ge4$, \eqref{QWE:ivp} admits a global small data smooth solution
(see \cite{Hormander97book,Klainerman80,KP83,Li-T});
when $n=2,3$ and the related null conditions hold, the global existence of $u$ has
been established in \cite{Alinhac01a} and \cite{Christodoulou86,Klainerman}, respectively;
otherwise, when $n=2,3$ and the null conditions are violated, the solution $u$ can blow up in finite time
(see \cite{Alinhac99, Alinhac01b, Hol}). In addition, for the global existence or blowup of small data smooth solutions to
the initial value problems of nonlinear wave equation systems, there have been also a lot of results, one can be
referred to the recent papers \cite{Dong-1}, \cite{Dong-2}, \cite{LM}, \cite{Speck} and so on.}
\end{remark}

\begin{remark}\label{Rem-5}
{\it It is pointed out that there have been remarkable results
for the global existence of low regularity solutions to the $n$-dimensional wave
maps system \eqref{Tao;wavemap} in  $\R^{1+n}$
($n\ge 2$). For examples, when the initial data $(u, \p_tu)(0,x)$ is small in the homogeneous Besov space $\dot{B}^{2,1}_{n/2}(\Bbb R^n)\times\dot{B}^{2,1}_{n/2-1}(\Bbb R^n)$, the author in \cite{Tataru98,Tataru01} has proved the global
solution of \eqref{Tao;wavemap} in $C([0, \infty), \dot{B}^{2,1}_{n/2}(\Bbb R^n))\cap C^1([0, \infty), \dot{B}^{2,1}_{n/2-1}(\Bbb R^n))$;
when $(u, \p_t u)(0,x)$ is small in the critical Sobolev space $\dot{H}^{n/2}(\Bbb R^n)\times\dot{H}^{n/2-1}(\Bbb R^n)$,
 the global solution of \eqref{Tao;wavemap} in $C([0, \infty), \dot{H}^{n/2}(\Bbb R^n))
 \cap C^1([0, \infty), \dot{H}^{n/2-1}(\Bbb R^n))$
is established in \cite{TaoIMRN,TaoCMP}.}
\end{remark}

\begin{remark}\label{Rem-6}
{\it For the 2-D semilinear wave equations $\ds\Box u=Q(\p u)$ with the cubic
nonlinearity $Q(\p u)=O(|\p u|^3)$ satisfying the null condition,
the author in \cite{Kubo13} has established the global existence of small data smooth solution $u$ for the exterior problem with
both Dirichlet and Neumann boundary conditions.
In \cite{Kubo13},  the following estimates (see (4.2)  with $m=0$ in \cite{Kubo14} and (57) with $\eta=0,\rho=1$  in \cite{Kubo13})
play crucial roles:
\begin{equation}\label{YHC0-5}
|w|\le C\sup_{(s,y)\in[0,t]\times\R^2}\w{s}^{1/2}|g(s,y)|
\end{equation}
and
\begin{equation}\label{Kubo13:dpw-0}
\w{x}^{1/2}\w{t-|x|}|\p w|\le C\ln(2+t+|x|)\sum_{|b|\le1}\sup_{(s,y)\in[0,t]\times\R^2}\w{s}|\p^bg(s,y)|,
\end{equation}
where $w$ solves the 2-D linear wave equation $\Box w=g(t,x)$ with $(w,\p_tw)|_{t=0}=(0,0)$.
Due to the lack of time decay of $w$ in \eqref{YHC0-5} and the appearance of the large factor $\ln(2+t+|x|)$
in \eqref{Kubo13:dpw-0}, the method in \cite{Kubo13} can not be applied to
our problem \eqref{NLW} since the cubic nonlinearity in \eqref{NLW}
contains the solution $u$ itself.
Thanks to some improved estimates in our previous paper \cite{HouYinYuan24},
we can show the crucial spacetime decay estimate \eqref{thm1:decay:c} in Theorem \ref{thm1},
which is one of the key points to solve the global exterior domain problem \eqref{NLW}.}
\end{remark}

\begin{remark}\label{Rem-7}
{\it Based on the following ``KSS estimate" for 3-D wave equation (see \cite[Prop 2.1]{KSS02jam} or (1.7)-(1.8)
on Page 190 of \cite{MetcalfeSogge06})
\begin{equation}\label{KSSestimate-01}
\begin{split}
&\quad(\ln(2+t))^{-1/2}\Big(\|\w{x}^{-1/2}\p v\|_{L_t^2L_x^2([0,t]\times\R^3)}+\|\w{x}^{-3/2}v\|_{L_t^2L_x^2([0,t]\times\R^3)}\Big)\\
&\ls\|\p v(0,x)\|_{L^2(\R^3)}+\int_0^t\|\Box v(s,x)\|_{L^2(\R^3)}ds,
\end{split}
\end{equation}
the authors in \cite{KSS02jam}-\cite{KSS04} obtain the almost global existence of small data solution
to 3-D quasilinear wave equation in exterior domain with Dirichlet boundary condition.
Note that the proof of the KSS estimate heavily relies on the strong Huygens principle.
However, the strong Huygens principle does not hold in two space dimensions.
Meanwhile, it is also pointed out that the multiplier used in \cite{MetcalfeSogge06} and \cite{Sterbenz05} fails when $n=2$
(see \cite[Page 485]{Hepd-Metc23}). We will establish the pointwise estimates \eqref{thm1:decay:a}-\eqref{thm1:decay:c} in Theorem \ref{thm1}
together with the energy estimates
instead of the KSS estimates to study the global solution problem of \eqref{NLW}.}
\end{remark}

We now give some reviews on the small data solution problem of second order wave equations in the exterior domains.
Let $u$ solve
\begin{equation}\label{QWE:ibvp}
\left\{
\begin{aligned}
&\Box u=Q(\p u, \p^2u),\hspace{2.6cm} (t,x)\in(0,+\infty)\times\cK,\\
&(u,\p_tu)(0,x)=(\ve u_0,\ve u_1)(x),\qquad x\in\cK,\\
&\text{The Dirichlet or Neumann boundary condition of $u$  on $\p\cK$,}
\end{aligned}
\right.
\end{equation}
where the nonlinearity $Q(\p u, \p^2u)=O(|\p u|^2+|\p^2u|^2)$ is at least quadratic in $(\p u,\p^2u)$.

When $n\ge6$, the authors in \cite{Shib-Tsuti86} show that \eqref{QWE:ibvp} has a global smooth small solution $u$
under the Dirichlet boundary value condition.
When $n=3$ and the related null condition holds, the global existence of small solution
to the Dirichlet boundary problem has been studied in \cite{Godin95,KSS00,KSS02jfa}.
When $n=3$ and the null condition fails, by making use of the well-known ``KSS estimate"
(see \cite[Prop 2.1]{KSS02jam}), the authors in \cite{KSS02jam,KSS04} obtain the almost global small data solution for the Dirichlet boundary problem of the semilinear and quasilinear wave equations, respectively.
In addition, when $n=3$, for the non-trapping obstacles and non-diagonal systems involving multiple wave speeds, the global existence of small solutions with Dirichlet
boundary condition has been established in \cite{MetcalfeSogge05}.
With respect to the Neumann boundary value problem in \eqref{QWE:ibvp}, the global small solution or almost
global solution are obtained for $n=3$ in \cite{Li-Yin18} or \cite{Yuan22}, respectively.

When $n=2$, for the semilinear wave equations $\ds\Box u=Q(\p u)$, if the cubic term $Q(\p u)=O(|\p u|^3)$ satisfies the null condition,
the author in \cite{Kubo13} has proved the global existence of small data smooth solutions for the exterior problems
with both Dirichlet and Neumann boundary conditions; if $\ds\Box u=Q(\p u)$ does not fulfill
the null condition, the almost global existence has been shown in \cite{KKL13,Kubo14}.
Recently,for $n=2$ and the quasilinear wave equation with quadratic $Q_0-$type null form,
in \cite{HouYinYuan24} we have established the global existence of small solution when
the Dirichlet boundary condition is imposed.

On the other hand, when $n=4$, for the initial boundary value problem of $\Box u=F(u,\p u,\p^2u)$
with the nonlinearity $F(u,\p u,\p^2u)$ depending on
both $u$ itself and its derivatives,
the authors in \cite{ZhaZhou15} prove the almost global existence of small data solutions to
the 4-D quasilinear wave equations outside of a star-shaped obstacle. In particular,
for $F(u,\p u,\p^2u)$ excluding the $u^2-$type nonlinearity, the
4-D quasilinear wave equation systems admit global small data smooth solutions (see \cite{MetcalfeMorgan20}).

\vskip 0.1 true cm

To the best of our knowledge, so far there are no global existence results on the small data smooth
solutions of 2-D wave maps equations in exterior domains.

Let us give some comments on the proof of Theorem \ref{thm1}.
Note that for the 2-D quasilinear wave equation with the quadratic $Q_0-$type null form
in the exterior domain, through looking for some good unknowns to transform
such a quadratically quasilinear wave equation into certain manageable cubic nonlinear form
and by deriving some crucial pointwise estimates for
the initial boundary value problem of 2-D linear wave equations in exterior domain,
we have established the global existence of small data solution in \cite{HouYinYuan24}.
However, since the wave maps equation in \eqref{WaveMap} includes the vector valued unknown function $u=(u^1,\dots,u^M)$,
it seems that the good unknowns introduced in \cite{HouYinYuan24} for the scalar quasilinear wave equation
are not applicable for problems \eqref{WaveMap} and \eqref{NLW}. To overcome this difficulty,
based on some precise estimates on the solution of the 2-D linear wave equation in \cite{HouYinYuan24},
we firstly  make use of the ghost weight energy method (the ghost weight is firstly introduced in \cite{Alinhac01a})
together with the Hardy inequality and elliptic estimates
to drive some rough pointwise estimates (see details in Sections \ref{sect3-1} and \ref{sect3-2} below).
More concretely, instead of the precise estimates \eqref{thm1:decay:a}-\eqref{thm1:decay:c} and \eqref{thm1:decay:LE},
we will establish the following weaker estimates (see Lemmas \ref{lem:4-1}-\ref{lem:4-4})
\begin{equation}\label{intro:decay}
\begin{split}
\sum_{|a|\le2N-33}|\p Z^au|&\le C\ve\w{x}^{-1/2}\w{t-|x|}^{-1}(1+t)^{29\ve_2},\quad\ve_2:=10^{-3},\\
\sum_{|a|\le2N-17}\sum_{i=1}^2|\bar\p_iZ^au|&\le C\ve\w{x}^{-1/2}\w{t+|x|}^{10\ve_2-1},\quad |x|\ge1+t/2,\\
\sum_{|a|\le2N-15}|Z^au|&\le C\ve\w{t+|x|}^{7\ve_2-1/2}\w{t-|x|}^{-\ve_2},\\
\sum_{|a|\le2N-13}\|\p^au\|_{L^2(\cK_R)}&\le C_R\ve(1+t)^{5\ve_2-1}.
\end{split}
\end{equation}
Secondly, in terms of \eqref{intro:decay},  we derive the crucial estimates \eqref{thm1:decay:a}-\eqref{thm1:decay:c}
and \eqref{thm1:decay:LE} by some more delicate energy methods (see Lemmas \ref{lem:4-5}-\ref{lem:4-8}).
Subsequently, the proof of Theorem \ref{thm1} is completed by the continuity argument.

\vskip 0.1 true cm

\noindent \textbf{Notations:}
\begin{itemize}
  \item $\cK:=\R^2\setminus\cO$, $\cK_R:=\cK\cap\{|x|\le R\}$, where $R>1$ is a fixed constant which may be
  changed in different places.
 \item  For the convenience and without loss of generality,
$\p\cK\subset\{x: c_0<|x|<1/2\}$ is assumed, where $c_0>0$ is some positive constant.
  \item The cutoff function $\chi_{[a,b]}(s)\in C^\infty(\R)$ with $0<a<b$, $0\le\chi_{[a,b]}(s)\le1$ and
  \begin{equation*}
    \chi_{[a,b]}(s)=\left\{
    \begin{aligned}
    0,\qquad |x|\le a,\\
    1,\qquad |x|\ge b.
    \end{aligned}
    \right.
  \end{equation*}
  \item $\w{x}:=\sqrt{1+|x|^2}$.
  \item $\N_0:=\{0,1,2,\cdots\}$.
  \item $\p_0:=\p_t$, $\p_1:=\p_{x_1}$, $\p_2:=\p_{x_2}$, $\p_x:=\nabla_x=(\p_1,\p_2)$, $\p:=(\p_t,\p_1,\p_2)$.
  \item For $|x|>0$, define $\bar\p_i:=\p_i+\frac{x_i}{|x|}\p_t$ ($i=1,2$) and $\bar\p=(\bar\p_1,\bar\p_2)$.
  \item $\Omega:=\Omega_{12}:=x_1\p_2-x_2\p_1$,  $Z=\{Z_1,Z_2,Z_3,Z_4\}:=\{\p_t,\p_1,\p_2,\Omega\}$, $\tilde Z:=\chi_{[1/2,1]}(x)Z$.
  \item $\p_x^a:=\p_1^{a_1}\p_2^{a_2}$ for $a\in\N_0^2$, $\p^a:=\p_t^{a_1}\p_1^{a_2}\p_2^{a_3}$ for $a\in\N_0^3$
      and $Z^a:=Z_1^{a_1}Z_2^{a_2}Z_2^{a_3}Z_4^{a_4}$ for $a\in\N_0^4$.
  \item For the quantities $f,g\ge0$, $f\ls g$ or $g\gt f$ denotes $f\le Cg$ for a generic constant $C>0$
  independent of $\ve$. In addition, $f\approx g$ means $f\ls g$ and $g\ls f$.
  \item $|\bar\p f|:=\sqrt{|\bar\p_1f|^2+|\bar\p_2f|^2}$.
  \item $\cW_{\mu,\nu}(t,x)=\w{t+|x|}^\mu(\min\{\w{x},\w{t-|x|}\})^\nu$ for $\mu, \nu\in\Bbb R$.
  \item For vector valued function $u=(u^1,\cdots,u^M)$, $\|u\|:=\sqrt{\|u^1\|^2+\cdots+\|u^M\|^2}$ with norm $\|\cdot\|=|\cdot|,\|\cdot\|_{L^2(\cK)},\|\cdot\|_{L^2(\cK_R)},\|\cdot\|_{H^k(\cK)}$.
\end{itemize}

\vskip 0.1 true cm

The paper is organized as follows.
In Section \ref{sect2}, some preliminaries and the bootstrap assumptions are given.
The required energy estimates are achieved by the ghost weight technique and the Hardy inequality
in Section \ref{sect3}.
In Section \ref{sect4}, some rough pointwise estimates established in Section \ref{sect3}
are improved and further the bootstrap assumptions are shown.
From this, the proof of Theorem \ref{thm1} can be completed.

\section{Preliminaries and bootstrap assumptions}\label{sect2}

\subsection{Null condition, Hardy inequality, elliptic estimate and Sobolev embedding}\label{sect2-1}

\begin{lemma}\label{lem:eqn:high}
Let $u$ be a smooth solution of \eqref{NLW}.
Then for any multi-index $a$, $Z^au$ satisfies
\begin{equation}\label{eqn:high}
\begin{split}
\Box Z^au^I&=\sum_{J,K,L=1}^MC_{IJKL}Z^a(u^J\cQ(u^K,u^L))\\
&=\sum_{b+c+d\le a}\sum_{J,K,L=1}^M\sum_{\cQ\in\{Q_0,Q_{\alpha\beta}\}}C_{IJKL}^{abcd}Z^bu^J\cQ(Z^cu^K,Z^du^L),
\end{split}
\end{equation}
where $C_{IJKL}^{abcd}$ are constants.
\end{lemma}
\begin{proof}
One can see Lemma 6.6.5 in \cite{Hormander97book}.
\end{proof}
\begin{lemma}[Null condition structure property]
For any smooth functions $f$ and $g$, it holds that
\begin{equation}\label{null:sturcture}
\sum_{\cQ\in\{Q_0,Q_{\alpha\beta}\}}|\cQ(f,g)|\ls|\bar\p f||\p g|+|\p f||\bar\p g|.
\end{equation}
\end{lemma}
\begin{proof}
Since the proof is analogous to that in Section 9.1 of \cite{Alinhac:book} and \cite[Lemma 2.2]{HouYin20jde}, we omit
the details here.
\end{proof}

\begin{lemma}[Elliptic estimate, Lemma 3.2 of \cite{Kubo13} or Lemma 3.1 of \cite{Kubo14}]
Let $w\in H^j(\cK)$ and $w|_{\cK}=0$ with integer $j\ge2$.
Then for any fixed constant $R>1$ and multi-index $a\in\N^2$ with $2\le|a|\le j$, one has
\begin{equation}\label{ellip}
\|\p_x^aw\|_{L^2(\cK)}\ls\|\Delta w\|_{H^{|a|-2}(\cK)}+\|w\|_{H^{|a|-1}(\cK_{R+1})}.
\end{equation}
\end{lemma}

\begin{lemma} [Hardy inequality, Lemma 1.2 of \cite{Lind}]
For any smooth function $f(t,x)$ on $\Bbb R_+\times \R^2$, it holds that
\begin{equation}\label{hardy:ineq}
\Big\|\frac{f}{\w{t-|x|}}\Big\|_{L^2(\R^2)}\ls\|\nabla_x f\|_{L^2(\R^2)},
\end{equation}
provided that $\supp_x f\subset\{|x|\le t+M_0\}$.
\end{lemma}

\begin{lemma}[Sobolev embedding, Lemma 3.1 of \cite{Kubo13} or Lemma 3.6 of \cite{Kubo14}]
Given any function $f(x)\in C_0^2(\overline{\cK})$, one has that for all $x\in\cK$,
\begin{equation}\label{Sobo:ineq}
\w{x}^{1/2}|f(x)|\ls\sum_{|a|\le2}\|Z^af\|_{L^2(\cK)}.
\end{equation}
\end{lemma}

\subsection{Two key lemmas}\label{sect2-2}

We now list two important lemmas established in \cite{HouYinYuan24}.
\begin{lemma} [Lemma 4.1 of \cite{HouYinYuan24}]
Suppose that the obstacle $\cO$ is star-shaped and $\cK=\R^2\setminus\cO$. Let $w$ be the solution of the IBVP
\begin{equation*}
\left\{
\begin{aligned}
&\Box w=F(t,x),\qquad(t,x)\in(0,\infty)\times\cK,\\
&w|_{\p\cK}=0,\\
&(w,\p_tw)(0,x)=(w_0,w_1)(x),\quad x\in\cK,
\end{aligned}
\right.
\end{equation*}
where $(w_0,w_1)$ has compact support and $\supp_x F(t,x)\subset \{|x|\le t+M_0\}$.
Then one has that for any $\mu,\nu\in(0,1/2)$ and $R>1$,
\begin{equation}\label{pw:crit}
\begin{split}
\w{t}\|w\|_{L^\infty(\cK_R)}&
\ls\|(w_0,w_1)\|_{H^2(\cK)}+\sum_{|a|\le1}\sup_{s\in[0,t]}\w{s}\|\p^aF(s,\cdot)\|_{L^2(\cK_3)}\\
&+\sum_{|a|\le1}\sup_{y\in\cK}|\p^aF(0,y)|
+\sum_{|a|\le2}\sup_{(s,y)\in[0,t]\times\cK}\w{y}^{1/2}\cW_{3/2+\mu,1+\nu}(s,y)|\p^aF(s,y)|
\end{split}
\end{equation}
and
\begin{equation}\label{pw:crit:dt}
\begin{split}
\frac{\w{t}}{\ln^2(2+t)}\|\p_tw\|_{L^\infty(\cK_R)}
\ls\|(w_0,w_1)\|_{H^3(\cK)}+\sum_{|a|\le2}\sup_{s\in[0,t]}\w{s}\|\p^aF(s,\cdot)\|_{L^2(\cK_3)}\\
+\sum_{|a|\le2}\sup_{y\in\cK}|\p^aF(0,y)|+\sum_{|a|\le3}\sup_{(s,y)\in[0,t]\times\cK}\w{y}^{1/2}\cW_{1,1}(s,y)|Z^aF(s,y)|.
\end{split}
\end{equation}
In addition, for $\mu,\nu\in(0,1/2)$,
\begin{equation}\label{pw:crit:dt'}
\begin{split}
\w{t}\|\p_tw\|_{L^\infty(\cK_R)}
&\ls\|(w_0,w_1)\|_{H^3(\cK)}+\sum_{|a|\le2}\sup_{s\in[0,t]}\w{s}\|\p^aF(s,\cdot)\|_{L^2(\cK_3)}\\
&+\sum_{|a|\le2}\sup_{y\in\cK}|\p^aF(0,y)|+\sum_{|a|\le3}\sup_{(s,y)\in[0,t]\times\cK}\w{y}^{1/2}\cW_{1+\mu+\nu,1}(s,y)|Z^aF(s,y)|.
\end{split}
\end{equation}
\end{lemma}

\begin{lemma}[Lemma 4.2 of \cite{HouYinYuan24}]
Suppose that the obstacle $\cO$ is star-shaped and $\cK=\R^2\setminus\cO$.
Let $w$ solve
\begin{equation*}
\left\{
\begin{aligned}
&\Box w=F(t,x),\qquad(t,x)\in(0,\infty)\times\cK,\\
&w|_{\p\cK}=0,\\
&(w,\p_tw)(0,x)=(w_0,w_1)(x),\quad x\in\cK,
\end{aligned}
\right.
\end{equation*}
where $(w_0,w_1)$ has compact support and $\supp_x F(t,x)\subset\{|x|\le t+M_0\}$.
Then it holds that for any $\mu,\nu\in(0,1/2)$,
\begin{equation}\label{InLW:pw}
\begin{split}
&\frac{\w{t+|x|}^{1/2}\w{t-|x|}^{\mu}}{\ln^2(2+t+|x|)}|w|\\
&\ls\|(w_0,w_1)\|_{H^5(\cK)}+\sum_{|a|\le4}\sup_{y\in\cK}|\p^aF(0,y)|
+\sum_{|a|\le4}\sup_{s\in[0,t]}\w{s}^{1/2+\mu}\|\p^aF(s,\cdot)\|_{L^2(\cK_4)}\\
&\quad+\sum_{|a|\le5}\sup_{(s,y)\in[0,t]\times(\overline{\R^2\setminus\cK_2})}\w{y}^{1/2}\cW_{1+\mu,1+\nu}(s,y)|\p^aF(s,y)|
\end{split}
\end{equation}
and
\begin{equation}\label{InLW:dpw}
\begin{split}
&\w{x}^{1/2}\w{t-|x|}|\p w|\\
&\ls\|(w_0,w_1)\|_{H^9(\cK)}+\sum_{|a|\le9}\sup_{y\in\cK}|\p^aF(0,y)|
+\sum_{|a|\le8}\sup_{s\in[0,t]}\w{s}\|\p^aF(s,\cdot)\|_{L^2(\cK_4)}\\
&\quad+\sum_{|a|\le9}\sup_{(s,y)\in[0,t]\times(\overline{\R^2\setminus\cK_2})}\w{y}^{1/2}\cW_{3/2+\mu,1+\nu}(s,y)|Z^aF(s,y)|.
\end{split}
\end{equation}
\end{lemma}

\subsection{Bootstrap assumptions}\label{sect2-3}

We make the following bootstrap assumptions for the solution $u$ of problem \eqref{NLW}:
\begin{align}
&\sum_{|a|\le N}|\p Z^au|\le\ve_1\w{x}^{-1/2}\w{t-|x|}^{-1},\label{BA1}\\
&\sum_{|a|\le N}|\bar\p Z^au|\le\ve_1\w{x}^{-1/2}\w{t+|x|}^{\ve_2-1},\label{BA2}\\
&\sum_{|a|\le N}|Z^au|\le\ve_1\w{t+|x|}^{\ve_2-1/2},\label{BA3}
\end{align}
where $\ve_2:=10^{-3}$ and $\ve_1\in(\ve,1)$ will be determined later.

Note that due to $\supp(u_0,u_1)\subset\{x: |x|\le M_0\}$,
the solution $u$ of problem \eqref{NLW} is supported in $\{x\in\cK:|x|\le t+M_0\}$.

\section{Energy estimates}\label{sect3}

\subsection{Energy estimates on the derivatives of solution}\label{sect3-1}

\begin{lemma}
Under the assumptions of Theorem \ref{thm1}, let $u$ be the solution of \eqref{NLW} and suppose that \eqref{BA1}-\eqref{BA3} hold.
Then there is a positive constant $C_0$ such that
\begin{equation}\label{energy:time}
\sum_{j\le2N}\|\p\p_t^ju\|_{L^2(\cK)}\ls\ve(1+t)^{C_0\ve_1},
\end{equation}
where $\ve_1>0$ is small enough. Especially, one has
\begin{equation}\label{energy:time'}
\sum_{j\le2N}\|\p\p_t^ju\|_{L^2(\cK)}\ls\ve(1+t)^{\ve_2},
\end{equation}
where and below $\ve_2=10^{-3}$.
\end{lemma}
\begin{proof}
Multiplying \eqref{eqn:high} by $e^q\p_tZ^au^I$ with the ghost weight $q=q(|x|-t)$ and $q'(s)=\w{s}^{-1.1}$ derives
\begin{equation}\label{energy:time1}
\begin{split}
&\quad\;\frac12\p_t[e^q(|\p_tZ^au^I|^2+|\nabla Z^au^I|^2)]-\dive(e^q\p_tZ^au^I\nabla Z^au^I)
+\frac{e^q}{2\w{t-|x|}^{1.1}}|\bar\p Z^au^I|^2\\
&=\sum_{b+c+d\le a}\sum_{J,K,L=1}^MC_{IJKL}^{abcd}e^q\p_tZ^au^IZ^bu^J\cQ(Z^cu^K,Z^du^L),
\end{split}
\end{equation}
where and below the summation $\ds\sum_{\cQ\in\{Q_0,Q_{\alpha\beta}\}}$ is omitted for convenience.

Set $Z^a=\p_t^j$ with $j=|a|$.
Integrating \eqref{energy:time1} over $[0,t]\times\cK$
by use of the boundary condition $\p_t^lu|_{\p\cK}=0$ for any integer $l\ge0$ and summing up $I$ from $1$ to $M$ yield
\begin{equation}\label{energy:time2}
\begin{split}
&\quad\;\|\p Z^au(t)\|^2_{L^2(\cK)}+\int_0^t\int_{\cK}\frac{|\bar\p Z^au(s,x)|^2}{\w{s-|x|}^{1.1}}dxds\\
&\ls\|\p Z^au(0)\|^2_{L^2(\cK)}+\sum_{b+c+d\le a}\int_0^t\int_{\cK}|\p_tZ^au|\cI^{bcd}dxds,
\end{split}
\end{equation}
where
\begin{equation}\label{energy:time3}
\cI^{bcd}:=\sum_{K,L=1}^M|Z^bu||\cQ(Z^cu^K,Z^du^L)|.
\end{equation}
Note that although $Z^a=\p_t^j$ is taken in $\cI^{bcd}$, one can analogously treat $\cI^{bcd}$
for $Z\in\{\p_t,\p_1,\p_2,\Omega\}$.
It follows from \eqref{null:sturcture} that
\begin{equation}\label{energy:time4}
\cI^{bcd}\ls|Z^bu||\bar\p Z^cu||\p Z^du|+|Z^bu||\p Z^cu||\bar\p Z^du|.
\end{equation}
When $|b|\ge N+1$, $Z^bu=\p_tZ^{b'}u$ with $|b'|=|b|-1$ and $|c|,|d|\le N$ hold.
Thus, it is deduced from \eqref{BA1}, \eqref{BA2} and \eqref{energy:time4} that
\begin{equation}\label{energy:time5}
\sum_{\substack{b+c+d\le a,\\|b|\ge N}}\int_{\cK}|\p_tZ^au|\cI^{bcd}dx\ls\ve_1(1+s)^{-1.9}\sum_{|b'|\le|a|}\|\p Z^{b'}u(s)\|^2_{L^2(\cK)}.
\end{equation}
With respect to the first term on the right hand side of \eqref{energy:time4}, when $|b|\le N$ and $|c|\le N$, by \eqref{BA2} and \eqref{BA3},
we have
\begin{equation}\label{energy:time6}
\sum_{\substack{b+c+d\le a,\\|b|,|c|\le N}}\int_{\cK}|\p_tZ^au||Z^bu||\bar\p Z^cu||\p Z^du|dx
\ls\ve_1(1+s)^{-1.4}\sum_{|d|\le|a|}\|\p Z^du(s)\|^2_{L^2(\cK)};
\end{equation}
when $|b|\le N$ and $|c|\ge N$, one has $|d|\le N$, and then it follows from
\eqref{BA1}, \eqref{BA3} and the Young inequality that
\begin{equation}\label{energy:time7}
\begin{split}
&\quad\sum_{\substack{b+c+d\le a,\\|b|\le N,|c|\ge N}}\int\int_{\cK}|\p_tZ^au||Z^bu||\bar\p Z^cu||\p Z^du|dxds\\
&\ls\ve_1\int_0^t\frac{\|\p Z^au(s)\|^2_{L^2(\cK)}ds}{(1+s)^{1.8}}
+\ve_1\sum_{|c|\le|a|}\int_0^t\int_{\cK}\frac{|\bar\p Z^cu(s,x)|^2}{\w{s-|x|}^{1.1}}dxds.
\end{split}
\end{equation}
With respect to the second term on the right hand side of \eqref{energy:time4}, we can obtain the same estimates as
in \eqref{energy:time6} and \eqref{energy:time7}.
Collecting \eqref{energy:time2}-\eqref{energy:time7} with all $|a|\le2N$ leads to
\begin{equation*}
\begin{split}
&\quad\sum_{|a|\le2N}\Big\{\|\p Z^au(t)\|^2_{L^2(\cK)}+\int_0^t\int_{\cK}\frac{|\bar\p Z^au(s,x)|^2}{\w{s-|x|}^{1.1}}dxds\Big\}\\
&\ls\ve^2+\ve_1\sum_{|a|\le2N}\Big\{\int_0^t\frac{\|\p Z^au(s)\|^2_{L^2(\cK)}ds}{(1+s)^{1.4}}
+\int_0^t\int_{\cK}\frac{|\bar\p Z^au(s,x)|^2}{\w{s-|x|}^{1.1}}dxds\Big\}.
\end{split}
\end{equation*}
Therefore, by $Z^a=\p_t^j$, it holds that for sufficiently small $\ve_1>0$,
\begin{equation*}
\sum_{j\le2N}\|\p\p_t^ju(t)\|^2_{L^2(\cK)}\ls\ve^2+\sum_{j\le2N}\ve_1\int_0^t\frac{\|\p\p_t^ju(s)\|^2_{L^2(\cK)}ds}{(1+s)^{1.4}}.
\end{equation*}
This, together with the Gronwall's lemma, yields \eqref{energy:time} and \eqref{energy:time'}.
\end{proof}

\begin{lemma}
Under the assumptions of Theorem \ref{thm1}, let $u$ be the solution of \eqref{NLW} and suppose that \eqref{BA1}-\eqref{BA3} hold.
Then we have
\begin{equation}\label{energy:dtdx}
\sum_{|a|\le2N}\|\p\p^au\|_{L^2(\cK)}\ls\ve_1(1+t)^{\ve_2}.
\end{equation}
\end{lemma}
\begin{proof}
Set $\ds E_j(t):=\sum_{k=0}^{2N-j}\sum_{I=1}^M\|\p\p_t^ku^I(t)\|_{H^j(\cK)}$ with $0\le j\le2N$.
Then for $j\ge1$, one has
\begin{equation}\label{energy:dtdx1}
\begin{split}
E_j(t)&\ls\sum_{k=0}^{2N-j}\sum_{I=1}^M[\|\p\p_t^ku^I(t)\|_{L^2(\cK)}
+\sum_{1\le|a|\le j}(\|\p_t\p_t^k\p_x^au^I(t)\|_{L^2(\cK)}+\|\p_x\p_t^k\p_x^au^I(t)\|_{L^2(\cK)})]\\
&\ls E_0(t)+E_{j-1}(t)+\sum_{k=0}^{2N-j}\sum_{I=1}^M\sum_{2\le|a|\le j+1}\|\p_t^k\p_x^au^I(t)\|_{L^2(\cK)},
\end{split}
\end{equation}
where we have used
\begin{equation*}
\sum_{k=0}^{2N-j}\sum_{1\le|a|\le j}\|\p_t\p_t^k\p_x^au^I(t)\|_{L^2(\cK)}
\ls\sum_{k=0}^{2N-j}\sum_{|a'|\le j-1}\|\p_x\p_t^{k+1}\p_x^{a'}u^I(t)\|_{L^2(\cK)}
\ls E_{j-1}(t).
\end{equation*}
For the last term in \eqref{energy:dtdx1}, it follows from \eqref{ellip} that
\begin{equation}\label{energy:dtdx2}
\begin{split}
\|\p_t^k\p_x^au^I\|_{L^2(\cK)}&\ls\|\Delta_x\p_t^ku^I\|_{H^{|a|-2}(\cK)}+\|\p_t^ku^I\|_{H^{|a|-1}(\cK_{R+1})}\\
&\ls\|\p_t^{k+2}u^I\|_{H^{|a|-2}(\cK)}+\|\p_t^k\Box u^I\|_{H^{|a|-2}(\cK)}+\|\p_t^ku^I\|_{L^2(\cK_{R+1})}\\
&\quad+\sum_{1\le|b|\le|a|-1\le j}\|\p_t^k\p_x^bu^I\|_{L^2(\cK_{R+1})},
\end{split}
\end{equation}
where the fact of  $\Delta=\p_t^2-\Box$ is utilized.
In addition, by \eqref{NLW} and \eqref{BA3} with $|a|+k\le2N+1$ we have
\begin{equation}\label{energy:dtdx3}
\|\p_t^k\Box u^I\|_{H^{|a|-2}(\cK)}
\ls\sum_{l\le k}\sum_{J=1}^M\ve_1\|\p\p_t^lu^J\|_{H^{|a|-2}(\cK)}\ls\ve_1E_{j-1}(t).
\end{equation}
Substituting \eqref{energy:dtdx2}-\eqref{energy:dtdx3} into \eqref{energy:dtdx1} together
with \eqref{BA3} and \eqref{energy:time'} yields
\begin{equation*}
E_j(t)\ls E_0(t)+E_{j-1}(t)+\sum_{I=1}^M\|u^I\|_{L^2(\cK_{R+1})}\ls\ve(1+t)^{\ve_2}+E_{j-1}(t)+\ve_1.
\end{equation*}
This implies \eqref{energy:dtdx}.
\end{proof}

\subsection{Energy estimates on the vector field derivatives of solution}\label{sect3-2}

\begin{lemma}
Under the assumptions of Theorem \ref{thm1}, let $u$ be the solution of \eqref{NLW} and suppose that \eqref{BA1}-\eqref{BA3} hold.
Then one has
\begin{equation}\label{energy:A}
\sum_{|a|\le2N-1}\|\p Z^au\|_{L^2(\cK)}\ls\ve_1(1+t)^{\ve_2+1/2}.
\end{equation}
\end{lemma}
\begin{proof}
Analogous to \eqref{energy:time2}, we have
\begin{equation}\label{energy:A1}
\begin{split}
&\quad\;\|\p Z^au(t)\|^2_{L^2(\cK)}+\int_0^t\int_{\cK}\frac{|\bar\p Z^au(s,x)|^2}{\w{s-|x|}^{1.1}}dxds\\
&\ls\ve^2+|\cB^a|+\sum_{b+c+d\le a}\int_0^t\int_{\cK}|\p_tZ^au|\cI^{bcd}dxds,
\end{split}
\end{equation}
where
\begin{equation}\label{energy:A2}
\cB^a:=\sum_{I=1}^M\int_0^t\int_{\p\cK}\nu(x)\cdot\nabla Z^au^I(s,x)\p_tZ^au^I(s,x)d\sigma ds
\end{equation}
with $\nu(x)=(\nu_1(x),\nu_2(x))$ being the unit outer normal of the boundary $\cK$ and $d\sigma$ being the curve measure on $\p\cK$.
By $\p\cK\subset\overline{\cK_1}$ and the trace theorem, one has
\begin{equation}\label{energy:A3}
\begin{split}
|\cB^a|&\ls\sum_{|b|\le|a|}\int_0^t\|(1-\chi_{[1,2]}(x))\p_t\p^bu\|_{L^2(\p\cK)}\|(1-\chi_{[1,2]}(x))\p_x\p^bu\|_{L^2(\p\cK)}ds\\
&\ls\sum_{|b|\le|a|}\int_0^t\|(1-\chi_{[1,2]}(x))\p \p^bu\|^2_{H^1(\cK)}ds\\
&\ls\sum_{|b|\le|a|+1}\int_0^t\|\p \p^bu\|^2_{L^2(\cK_2)}ds\ls\ve_1^2(1+t)^{2\ve_2+1},
\end{split}
\end{equation}
where \eqref{energy:dtdx} is used.

When $|b|\le N$ in $\cI^{bcd}$ given by \eqref{energy:time3}, analogously to the estimates of \eqref{energy:time6}
and \eqref{energy:time7}, we have
\begin{equation}\label{energy:A4}
\begin{split}
&\sum_{\substack{b+c+d\le a,\\|b|\le N}}\int_0^t\int_{\cK}|\p_tZ^au|\cI^{bcd}dxds
\ls\ve_1\sum_{|b|\le|a|}\int_0^t\int_{\cK}\frac{|\bar\p Z^bu(s,x)|^2}{\w{s-|x|}^{1.1}}dxds\\
&\qquad\qquad+\ve_1\sum_{|b|\le|a|}\int_0^t\frac{\|\p Z^bu(s)\|^2_{L^2(\cK)}ds}{(1+s)^{1.4}}.
\end{split}
\end{equation}
Next, we turn to the estimate of $\cI^{bcd}$ with $|b|\ge N+1$.
In this case, $|c|,|d|\le N$ hold.
Due to $Z\in\{\p,\Omega\}$, then for any function $f$,
\begin{equation}\label{energy:A5}
|Zf|\ls\w{x}|\p f|.
\end{equation}
Denote $\ds\tilde\chi(\frac{|x|}{t+9}):=\chi_{[1/3,1/2]}(\frac{|x|}{t+9})$.
It follows from \eqref{hardy:ineq}, \eqref{BA1}, \eqref{BA2}, \eqref{energy:time3} and \eqref{energy:A5} that
\begin{equation}\label{energy:A6}
\begin{split}
&\sum_{\substack{b+c+d\le a,\\|b|\ge N+1}}\|\cI^{bcd}\|_{L^2(\cK)}
\ls\sum_{1\le|b|\le|a|}\ve_1^2\w{t}^{\ve_2-1}\Big\|\frac{|Z^bu|}{\w{x}\w{t-|x|}}\Big\|_{L^2(\cK)}\\
&\ls\sum_{1\le|b|\le|a|}\ve_1^2\w{t}^{\ve_2-2}\Big(\Big\|\frac{\tilde\chi|Z^bu|}{\w{t-|x|}}\Big\|_{L^2(\R^2)}
+\Big\|\frac{(1-\tilde\chi)|Z^bu|}{\w{x}}\Big\|_{L^2(\cK)}\Big)\\
&\ls\ve_1^2\w{t}^{\ve_2-2}(\sum_{|b|\le|a|}\|\p Z^bu\|_{L^2(\cK)}+\sum_{1\le|b|\le|a|}\w{t}^{-1}\Big\|\tilde\chi'|Z^bu|\Big\|_{L^2(\cK)})\\
&\ls\sum_{|b|\le|a|}\ve_1^2\w{t}^{\ve_2-2}\|\p Z^bu\|_{L^2(\cK)}.
\end{split}
\end{equation}
Collecting \eqref{energy:A1}-\eqref{energy:A6} with all $|a|\le2N-1$ implies that
\begin{equation*}
\begin{split}
&\sum_{|a|\le2N-1}\Big\{\|\p Z^au(t)\|^2_{L^2(\cK)}+\int_0^t\int_{\cK}\frac{|\bar\p Z^au(s,x)|^2}{\w{s-|x|}^{1.1}}dxds\Big\}
\ls\ve^2+\ve_1^2(1+t)^{2\ve_2+1}\\
&\qquad+\ve_1\sum_{|a|\le2N-1}\int_0^t\int_{\cK}\frac{|\bar\p Z^au(s,x)|^2}{\w{s-|x|}^{1.1}}dxds
+\ve_1\sum_{|a|\le2N-1}\int_0^t\frac{\|\p Z^au(s)\|^2_{L^2(\cK)}ds}{(1+s)^{1.4}},
\end{split}
\end{equation*}
which yields \eqref{energy:A}.
\end{proof}

\subsection{Decay estimates of the local energy and improved energy estimates}\label{sect:LocEnergy}

To improve \eqref{energy:A}, one requires a better estimate on the boundary term \eqref{energy:A3}.
To this end, we first derive the precise time decay estimates of the local energy for problem \eqref{NLW}
by the spacetime pointwise estimate and the elliptic estimate.
We now treat the local energy $\|u\|_{L^2(\cK_R)}$ by utilizing \eqref{pw:crit}.
\begin{lemma}
Under the assumptions of Theorem \ref{thm1}, let $u$ be the solution of \eqref{NLW} and suppose that \eqref{BA1}-\eqref{BA3} hold.
Then
\begin{equation}\label{u:loc:pw}
\|u\|_{L^\infty(\cK_R)}\ls(\ve+\ve_1^2)(1+t)^{-1}.
\end{equation}
\end{lemma}
\begin{proof}

Applying \eqref{pw:crit} with $\mu=\nu=\ve_2$ to \eqref{NLW} yields
\begin{equation}\label{u:loc:pw1}
\begin{split}
\w{t}&\|u\|_{L^\infty(\cK_R)}\ls\ve+\sum_{\substack{J,K,L=1\cdots,M,\\|a|\le1}}\sup_{s\in[0,t]}\w{s}\|\p^a(u^J\cQ(u^K,u^L))(s,\cdot)\|_{L^2(\cK_3)}\\
&+\sum_{\substack{J,K,L=1\cdots,M,\\|a|\le2}}\sup_{(s,y)\in[0,t]\times\cK}\w{y}^{1/2}\cW_{3/2+\ve_2,1+\ve_2}(s,y)|\p^a(u^J\cQ(u^K,u^L))(s,y)|,
\end{split}
\end{equation}
where the initial data condition \eqref{initial:data} is used.
It follows from \eqref{null:sturcture}, \eqref{BA1}, \eqref{BA2}, \eqref{BA3} and $N\ge42$ that
\begin{equation*}
\sum_{|a|\le2}|\p^a(u^J\cQ(u^K,u^L))(s,y)|\ls\ve_1^3\w{y}^{-1}\w{s-|y|}^{-1}\w{s+|y|}^{2\ve_2-3/2}.
\end{equation*}
This, together with \eqref{u:loc:pw1}, yields \eqref{u:loc:pw}.
\end{proof}

\begin{lemma}
Under the assumptions of Theorem \ref{thm1}, let $u$ be the solution of \eqref{NLW} and suppose that \eqref{BA1}-\eqref{BA3} hold.
Then we have
\begin{equation}\label{LocEnergydt}
\sum_{j\le2N-6}\|\p_t\p_t^ju\|_{L^2(\cK_R)}\ls(\ve+\ve_1^2)(1+t)^{2\ve_2-1/2}\ln^2(2+t).
\end{equation}
\end{lemma}
\begin{proof}
Applying \eqref{pw:crit:dt} to \eqref{eqn:high} with $Z^a=\p_t^j$, $j\le2N-6$ and \eqref{initial:data} leads to
\begin{equation}\label{LocEnergydt1}
\begin{split}
\frac{\w{t}}{\ln^2(2+t)}\sum_{j\le2N-6}\|\p_t\p_t^ju\|_{L^\infty(\cK_R)}
\ls\ve+\sum_{\substack{J,K,L=1\cdots,M,\\|a|\le2N-4}}\sup_{s\in[0,t]}\w{s}\|\p^a(u^J\cQ(u^K,u^L))(s,\cdot)\|_{L^2(\cK_3)}\\
+\sum_{|a|\le3}\sum_{\substack{J,K,L=1\cdots,M,\\j\le2N-6}}\sup_{(s,y)\in[0,t]\times\cK}\w{y}^{1/2}\cW_{1,1}(s,y)|Z^a\p_t^j(u^J\cQ(u^K,u^L))(s,y)|.
\end{split}
\end{equation}
By \eqref{Sobo:ineq} and \eqref{energy:A}, we can get
\begin{equation}\label{LocEnergydt2}
\sum_{|a|\le2N-3}|\p Z^au(t,x)|\ls\ve_1\w{x}^{-1/2}(1+t)^{\ve_2+1/2}.
\end{equation}
Then it follows from \eqref{BA1}, \eqref{BA3} and \eqref{LocEnergydt2} that
\begin{equation}\label{LocEnergydt3}
\begin{split}
&\sum_{|a|\le3}\sum_{j\le2N-6}|Z^a\p_t^j(u^J\cQ(u^K,u^L))(s,y)|
\ls\sum_{|b|\le3}\sum_{k+|b|+|c|+|d|\le2N-6}|\p_t^kZ^bu||\p Z^cu||\p Z^du|\\
&\ls\ve_1^3\w{s+|y|}^{\ve_2-1/2}\w{y}^{-1}(1+s)^{\ve_2+1/2}\w{s-|y|}^{-1}
+\ve_1^3\w{y}^{-3/2}(1+s)^{\ve_2+1/2}\w{s-|y|}^{-2}.
\end{split}
\end{equation}
Substituting \eqref{LocEnergydt3} into \eqref{LocEnergydt1} yields
\begin{equation*}
\sum_{j\le2N-6}\|\p_t\p_t^ju\|_{L^\infty(\cK_R)}\ls(\ve+\ve_1^3)(1+t)^{2\ve_2-1/2}\ln^2(2+t),
\end{equation*}
which completes the proof of \eqref{LocEnergydt}.
\end{proof}

\begin{lemma}
Under the assumptions of Theorem \ref{thm1}, let $u$ be the solution of \eqref{NLW} and suppose that \eqref{BA1}-\eqref{BA3} hold.
Then one has
\begin{equation}\label{loc:energy}
\sum_{|a|\le2N-6}\|\p^au\|_{L^2(\cK_R)}\ls(\ve+\ve_1^2)(1+t)^{3\ve_2-1/2}.
\end{equation}
\end{lemma}
\begin{proof}
For $j=0,1,2,\cdots,2N-6$, set $\ds E_j^{loc}(t):=\sum_{k\le j}\sum_{I=1}^M\|\p_t^k\p_x^{2N-6-j}u^I\|_{L^2(\cK_{R+j})}$.
Note that for $j\le2N-8$, $\p_x^{2N-6-j}=\p_x^2\p_x^{2N-8-j}$.
Then applying \eqref{ellip} to $(1-\chi_{[R+j,R+j+1]})\p_t^{k}u^I$ leads to
\begin{equation}\label{loc:energy1}
\begin{split}
E_j^{loc}(t)&\ls\sum_{k\le j}\sum_{I=1}^M\|\p_x^{2N-6-j}(1-\chi_{[R+j,R+j+1]})\p_t^ku^I\|_{L^2(\cK)}\\
&\ls\sum_{k\le j}\sum_{I=1}^M\|\Delta(1-\chi_{[R+j,R+j+1]})\p_t^ku^I\|_{H^{2N-8-j}(\cK)}\\
&\quad+\sum_{k\le j}\sum_{I=1}^M\|(1-\chi_{[R+j,R+j+1]})\p_t^ku^I\|_{H^{2N-7-j}(\cK)}\\
&\ls\sum_{k\le j}\sum_{I=1}^M\|(1-\chi_{[R+j,R+j+1]})\Delta\p_t^ku^I\|_{H^{2N-8-j}(\cK)}+\sum_{j+1\le l\le2N-6}E_l^{loc}(t)\\
&\ls\sum_{k\le j}\sum_{I=1}^M\|\Box\p_t^ku^I\|_{H^{2N-8-j}(\cK_{R+j+1})}
+\sum_{k\le j}\sum_{I=1}^M\|\p_t^{k+2}u^I\|_{H^{2N-8-j}(\cK_{R+j+1})}\\
&\quad+\sum_{j+1\le l\le2N-6}E_l^{loc}(t),
\end{split}
\end{equation}
where the fact of $\Delta=\p_t^2-\Box$ is applied.
It follows from \eqref{NLW}, \eqref{BA1}, \eqref{BA3} and \eqref{energy:dtdx} that for $j\le2N-8$,
\begin{equation}\label{loc:energy2}
\sum_{k\le j}\|\Box\p_t^ku^I\|_{H^{2N-8-j}(\cK_{R+j+1})}\ls\ve_1^2(1+t)^{-1.4}.
\end{equation}
Substituting \eqref{loc:energy2} into \eqref{loc:energy1} yields that for $j\le2N-8$,
\begin{equation}\label{loc:energy3}
E_j^{loc}(t)\ls\ve_1^2(1+t)^{-1}+\sum_{j+1\le l\le2N-6}E_l^{loc}(t).
\end{equation}
Next we estimate $E_{2N-7}^{loc}(t)$. Note that
\begin{equation}\label{loc:energy4}
(E_{2N-7}^{loc}(t))^2\ls\sum_{k\le2N-7}\sum_{i=1}^2\sum_{I=1}^M\|\p_i[(1-\chi_{[R+2N-7,R+2N-6]})\p_t^ku^I]\|^2_{L^2(\cK)}.
\end{equation}
Denote $\tilde\chi:=1-\chi_{[R+2N-7,R+2N-6]}$. It follows from the integration by parts
together with the boundary condition in \eqref{NLW} that
\begin{equation}\label{loc:energy5}
\begin{split}
&\quad\sum_{i=1}^2\|\p_i[(1-\chi_{[R+2N-7,R+2N-6]})\p_t^ku^I]\|^2_{L^2(\cK)}\\
&=\int_{\cK}\sum_{i=1}^2\p_i[\tilde\chi\p_t^ku^I\p_i(\tilde\chi\p_t^ku^I)]dx
-\int_{\cK}\tilde\chi\p_t^ku^I\Delta(\tilde\chi\p_t^ku^I)dx\\
&=-\int_{\cK}\tilde\chi^2\p_t^ku^I(\Delta\p_t^ku^I)dx-\int_{\cK}\tilde\chi(\Delta\tilde\chi)|\p_t^ku^I|^2dx
-2\int_{\cK}\tilde\chi\p_t^ku^I\nabla\tilde\chi\cdot\nabla\p_t^ku^Idx\\
&=-\int_{\cK}\tilde\chi^2\p_t^ku^I\p_t^{k+2}u^Idx+\int_{\cK}\tilde\chi^2\p_t^ku^I\Box\p_t^ku^Idx
-\int_{\cK}\tilde\chi(\Delta\tilde\chi)|\p_t^ku^I|^2dx\\
&\quad-\int_{\cK}\dive(\tilde\chi\nabla\tilde\chi|\p_t^ku^I|^2)dx
+\int_{\cK}\dive(\tilde\chi\nabla\tilde\chi)|\p_t^ku^I|^2dx\\
&\ls(E_{2N-6}^{loc}(t))^2+\|\p_t^{2N-5}u^I\|^2_{L^2(\cK_{R+2N})}+\ve_1^4(1+t)^{-2},
\end{split}
\end{equation}
where we have used $k\le2N-7$, $\Delta=\p_t^2-\Box$ and the estimate of $\Box\p_t^ku^I$ like \eqref{loc:energy2}.
On the other hand, by \eqref{u:loc:pw} and \eqref{LocEnergydt}, one has that
\begin{equation}\label{loc:energy6}
\begin{split}
E_{2N-6}^{loc}(t)&\ls\sum_{j\le2N-5}\sum_{I=1}^M\|\p_t^ju^I\|_{L^2(\cK_{R+2N})}\\
&\ls\sum_{I=1}^M\Big(\sum_{j\le2N-6}\|\p_t\p_t^ju^I\|_{L^2(\cK_{R+2N})}+\|u^I\|_{L^2(\cK_{R+2N})}\Big)\\
&\ls(\ve+\ve_1^2)(1+t)^{\ve_2-1/2}\ln^2(2+t)+(\ve+\ve_1^2)(1+t)^{-1}.
\end{split}
\end{equation}
Therefore, \eqref{loc:energy} can be achieved by \eqref{LocEnergydt}, \eqref{loc:energy3}, \eqref{loc:energy4}, \eqref{loc:energy5} and \eqref{loc:energy6}.
\end{proof}

\begin{lemma}
Under the assumptions of Theorem \ref{thm1}, let $u$ be the solution of \eqref{NLW} and suppose that \eqref{BA1}-\eqref{BA3} hold.
Then
\begin{equation}\label{energy:B}
\sum_{|a|\le2N-8}\|\p Z^au\|_{L^2(\cK)}\ls\ve_1(1+t)^{3\ve_2}.
\end{equation}
\end{lemma}
\begin{proof}
The proof of \eqref{energy:B} is similar to that of \eqref{energy:A} with a better estimate than \eqref{energy:A3}
for the boundary term $\cB^a$ defined in \eqref{energy:A2}.
In fact, by applying the same argument as in \eqref{energy:A3} and utilizing \eqref{loc:energy}, one has
\begin{equation}\label{energy:B1}
\sum_{|a|\le2N-8}|\cB^a|\ls\sum_{|b|\le|a|+1\le2N-7}\int_0^t\|\p \p^bu(s)\|^2_{L^2(\cK_2)}ds\ls\ve_1^2\int_0^t(1+s)^{6\ve_2-1}ds\ls\ve_1^2(1+t)^{6\ve_2}.
\end{equation}
Substituting \eqref{energy:B1} into \eqref{energy:A1} instead of \eqref{energy:A3} yields \eqref{energy:B}.
\end{proof}

\section{Improved pointwise estimates and proof of Theorem \ref{thm1}}\label{sect4}

\subsection{Decay estimates on the good derivatives of solution}

It is pointed out that the estimate \eqref{BA3} of $Z^au$ will play an essential role in the decay estimates
of the good derivatives \eqref{BA2}.
For this purpose, we firstly improve the decay estimate of the local energy \eqref{loc:energy} in Section \ref{sect:LocEnergy}
by utilizing \eqref{energy:B} instead of \eqref{energy:A}.
\begin{lemma}\label{lem:4-1}
Under the assumptions of Theorem \ref{thm1}, let $u$ be the solution of \eqref{NLW} and suppose that \eqref{BA1}-\eqref{BA3} hold.
Then we have
\begin{equation}\label{loc:improv}
\sum_{|a|\le2N-13}\|\p^au\|_{L^2(\cK_R)}\ls(\ve+\ve_1^2)(1+t)^{5\ve_2-1}.
\end{equation}
\end{lemma}
\begin{proof}
The proof procedure is similar to that in Section \ref{sect:LocEnergy} except deriving a better estimate than \eqref{LocEnergydt3}.

At first, combining \eqref{Sobo:ineq} and \eqref{energy:B} leads to
\begin{equation}\label{loc:improv1}
\sum_{|a|\le2N-10}|\p Z^au|\ls\ve_1\w{x}^{-1/2}(1+t)^{3\ve_2}.
\end{equation}
With the estimate \eqref{loc:improv1} instead of \eqref{LocEnergydt2}, the estimate \eqref{LocEnergydt3} can be improved into
\begin{equation*}
\sum_{|a|\le3}\sum_{j\le2N-13}|Z^a\p_t^j(u^J\cQ(u^K,u^L))(s,y)|
\ls\ve_1^3\w{s+|y|}^{\ve_2-1/2}\w{y}^{-1}(1+s)^{3\ve_2}\w{s-|y|}^{-1},
\end{equation*}
which yields
\begin{equation*}
\sum_{j\le2N-13}\|\p_t\p_t^ju\|_{L^2(\cK_R)}\ls(\ve+\ve_1^2)(1+t)^{4\ve_2-1}\ln^2(2+t).
\end{equation*}
Therefore, \eqref{loc:improv} can be obtained.
\end{proof}

\begin{lemma}
Under the assumptions of Theorem \ref{thm1}, let $u$ be the solution of \eqref{NLW} and suppose that \eqref{BA1}-\eqref{BA3} hold.
Then
\begin{equation}\label{BA3:improv}
\sum_{|a|\le2N-15}|Z^au|\ls(\ve+\ve_1^2)\w{t+|x|}^{7\ve_2-1/2}\w{t-|x|}^{-\ve_2}.
\end{equation}
\end{lemma}
\begin{proof}
Set
\begin{equation}\label{BA3:improv1}
\tilde Z:=\chi_{[1/2,1]}(x)Z
\end{equation}
and $\ds\tilde Z^au^I|_{\p\cK}=0$.
Applying \eqref{InLW:pw} to $\Box\tilde Z^au^I=\Box Z^au^I+\Box(\tilde Z^a-Z^a)u^I$ for $|a|\le2N-15$
and $\mu=\nu=\ve_2$ yields
\begin{equation}\label{BA3:improv2}
\begin{split}
&\frac{\w{t+|x|}^{1/2}\w{t-|x|}^{\ve_2}}{\ln^2(2+t+|x|)}|\tilde Z^au^I(t,x)|\\
&\ls\ve+\sum_{|b|\le4}\sup_{s\in[0,t]}\w{s}^{1/2+\ve_2}(\|\p^b\Box(\tilde Z^a-Z^a)u^I(s,\cdot)\|_{L^2(\cK_4)}
+\|\p^b\Box Z^au^I(s,\cdot)\|_{L^2(\cK_4)})\\
&\quad +\sum_{|b|\le5}\sup_{(s,y)\in[0,t]\times(\overline{\R^2\setminus\cK_2})}\w{y}^{1/2}\cW_{1+\ve_2,1+\ve_2}(s,y)|\p^b\Box Z^au^I(s,y)|,
\end{split}
\end{equation}
where we have used the facts of $\supp_x(\tilde Z^a-Z^a)u^I\subset\{|x|\le1\}$ and $\tilde Z=Z$ for $|x|\ge1$.
For the second line of \eqref{BA3:improv2}, it follows from \eqref{eqn:high}, \eqref{BA3} and \eqref{loc:energy}
that for $|a|\le2N-15$,
\begin{equation}\label{BA3:improv3}
\sum_{|b|\le4}\sup_{s\in[0,t]}\w{s}^{1/2+\ve_2}(\|\p^b\Box(\tilde Z^a-Z^a)u^I(s,\cdot)\|_{L^2(\cK_4)}
+\|\p^b\Box Z^au^I(s,\cdot)\|_{L^2(\cK_4)})
\ls(\ve+\ve_1^2)(1+t)^{4\ve_2}.
\end{equation}
Next, we treat the third line of \eqref{BA3:improv2}.
From \eqref{eqn:high} and \eqref{null:sturcture}, one has
\begin{equation}\label{BA3:improv4}
\sum_{|a|\le2N-15}\sum_{|b|\le5}|\p^b\Box Z^au^I|\ls\sum_{|b|+|c|+|d|\le2N-10}|Z^bu||\bar\p Z^cu||\p Z^du|.
\end{equation}
When $|b|\le N$ holds on the right hand side of \eqref{BA3:improv4}, then by \eqref{BA1}, \eqref{BA3} and \eqref{loc:improv1},
we arrive at
\begin{equation}\label{BA3:improv5}
\begin{split}
\sum_{\substack{|b|+|c|+|d|\le2N-10,\\|b|\le N}}|Z^bu||\bar\p Z^cu||\p Z^du|
&\ls\sum_{\substack{|b|+|c|+|d|\le2N-10,\\|b|\le N}}|Z^bu||\p Z^cu||\p Z^du|\\
&\ls\ve_1^3\w{s+|y|}^{4\ve_2-1/2}\w{y}^{-1}\w{s-|y|}^{-1}.
\end{split}
\end{equation}
When $|b|\ge N+1$, it can be deduced from \eqref{BA1}, \eqref{BA2}, \eqref{energy:A5} and \eqref{loc:improv1} that
\begin{equation}\label{BA3:improv6}
\sum_{\substack{|b|+|c|+|d|\le2N-10,\\|b|\ge N+1}}|Z^bu||\bar\p Z^cu||\p Z^du|
\ls\ve_1^3\w{y}^{-1/2}\w{s+|y|}^{4\ve_2-1}\w{s-|y|}^{-1}.
\end{equation}
Collecting \eqref{BA3:improv2}-\eqref{BA3:improv6} leads to
\begin{equation*}
\sum_{|a|\le2N-15}|\tilde Z^au^I(t,x)|\ls(\ve+\ve_1^2)\w{t+|x|}^{7\ve_2-1/2}\w{t-|x}^{-\ve_2}.
\end{equation*}
This, together with \eqref{loc:improv}, $R>1$, \eqref{BA3:improv1} and the Sobolev embedding,
yields \eqref{BA3:improv}.
\end{proof}

\begin{lemma}
Under the assumptions of Theorem \ref{thm1}, let $u$ be the solution of \eqref{NLW} and suppose that \eqref{BA1}-\eqref{BA3} hold.
Then one has that for $|x|\ge1+t/2$,
\begin{equation}\label{BA2:improv}
\sum_{|a|\le2N-17}|\bar\p Z^au|\ls(\ve+\ve_1^2)\w{t+|x|}^{10\ve_2-3/2}.
\end{equation}
\end{lemma}
\begin{proof}
Firstly, recall (6.20) and (6.22) in \cite{HouYinYuan24} (see also Section 4.6 in \cite{Kubo13}) that
\begin{equation}\label{BA2:improv1}
\begin{split}
&\quad\;\bar\p_1=\frac{x_1}{r}\p_+-\frac{x_2}{r^2}\Omega,
\quad\bar\p_2=\frac{x_2}{r}\p_++\frac{x_1}{r^2}\Omega,\quad r=|x|,\\
&\quad\;\p_+(r^{1/2}w)(t,r\frac{x}{|x|})-\p_+(r^{1/2}w)(0,(r+t)\frac{x}{|x|})\\
&=\int_0^t\{(r+t-s)^{1/2}\Box w+(r+t-s)^{-3/2}(w/4+\Omega^2w)\}(s,(r+t-s)\frac{x}{|x|})ds.
\end{split}
\end{equation}
By choosing $w=Z^au^I$ with $|a|\le2N-17$ in \eqref{BA2:improv1},
it follows from \eqref{eqn:high}, \eqref{BA1}, \eqref{loc:improv1} and \eqref{BA3:improv} that for $|y|\ge1+s/2$,
\begin{equation}\label{BA2:improv2}
|\Box Z^au^I(s,y)|\ls\ve_1^3\w{s+|y|}^{7\ve_2-1/2}(1+s)^{3\ve_2}\w{y}^{-1}\w{s-|y|}^{-1-\ve_2}.
\end{equation}
Substituting \eqref{initial:data}, \eqref{BA3:improv} and \eqref{BA2:improv2} into \eqref{BA2:improv1} yields that for $|x|\ge1+t/2$,
\begin{equation}\label{BA2:improv3}
\begin{split}
|\p_+(r^{1/2}Z^au^I)(t,x)|&\ls\ve\w{x}^{-1}+\ve_1^3(r+t)^{10\ve_2-1}\int_0^t(1+|r+t-2s|)^{-1-\ve_2}ds\\
&\quad+(\ve+\ve_1^2)(r+t)^{7\ve_2-1/2}\int_0^t(r+t-s)^{-3/2}ds\\
&\ls(\ve+\ve_1^2)\w{t+|x|}^{10\ve_2-1},
\end{split}
\end{equation}
which leads to
\begin{equation*}
\begin{split}
|\p_+Z^au^I|&\ls r^{-1/2}|\p_+(r^{1/2}Z^au^I)(t,x)|+r^{-1}|Z^au^I(t,x)|\\
&\ls(\ve+\ve_1^2)\w{t+|x|}^{10\ve_2-3/2}.
\end{split}
\end{equation*}
This, together with \eqref{BA3:improv} and \eqref{BA2:improv1}, derives \eqref{BA2:improv}.
\end{proof}

\subsection{Crucial pointwise estimates}

\begin{lemma}\label{lem:4-4}
Under the assumptions of Theorem \ref{thm1}, let $u$ be the solution of \eqref{NLW} and suppose that \eqref{BA1}-\eqref{BA3} hold.
Then
\begin{equation}\label{BA1:improv}
\sum_{|a|\le2N-33}|\p Z^au|\ls(\ve+\ve_1^2)\w{x}^{-1/2}\w{t-|x|}^{-1}(1+t)^{29\ve_2}.
\end{equation}
\end{lemma}
\begin{proof}
At first, we show that
\begin{equation}\label{BA1:improv1}
\sum_{|a|\le2N-24}|\p Z^au|\ls(\ve+\ve_1^2)\w{x}^{-1/2}\w{t-|x|}^{-1}(1+t)^{11\ve_2+1/2}.
\end{equation}
Applying \eqref{InLW:dpw} to $\tilde Z^au^I$ for $|a|\le2N-24$ and $\mu=\nu=\ve_2$ yields
\begin{equation}\label{BA1:improv2}
\begin{split}
&\quad\w{x}^{1/2}\w{t-|x|}|\p\tilde Z^au^I|\\
&\ls\ve+\sum_{|b|\le8}\sup_{s\in[0,t]}\w{s}(\|\p^b\Box(\tilde Z^a-Z^a)u^I(s,\cdot)\|_{L^2(\cK_4)}+\|\p^b\Box Z^au^I(s,\cdot)\|_{L^2(\cK_4)})\\
&+\sum_{|b|\le9}\sup_{(s,y)\in[0,t]\times(\overline{\R^2\setminus\cK_2})}\w{y}^{1/2}\cW_{3/2
+\ve_2,1+\ve_2}(s,y)|Z^b\Box Z^au^I(s,y)|,
\end{split}
\end{equation}
where we have used \eqref{initial:data}.
It follows from \eqref{eqn:high} and \eqref{loc:improv} that for $|a|\le2N-24$,
\begin{equation}\label{BA1:improv3}
\sum_{|b|\le8}\sup_{s\in[0,t]}\w{s}(\|\p^b\Box(\tilde Z^a-Z^a)u^I(s,\cdot)\|_{L^2(\cK_4)}
+\|\p^b\Box Z^au^I(s,\cdot)\|_{L^2(\cK_4)})
\ls(\ve+\ve_1^2)(1+t)^{5\ve_2}.
\end{equation}
On the other hand, collecting \eqref{eqn:high}, \eqref{BA1}, \eqref{loc:improv1} and \eqref{BA3:improv} leads to
\begin{equation}\label{BA1:improv4}
\begin{split}
\sum_{|a|\le2N-24}\sum_{|b|\le9}|Z^b\Box Z^au^I(s,y)|&\ls\sum_{|b|+|c|+|d|\le2N-15}|Z^bu||\p Z^cu||\p Z^du|\\
&\ls\ve_1^3\w{s+|y|}^{7\ve_2-1/2}(1+s)^{3\ve_2}\w{y}^{-1}\w{s-|y|}^{-1-\ve_2}.
\end{split}
\end{equation}
Substituting \eqref{BA1:improv3} and \eqref{BA1:improv4} into \eqref{BA1:improv2} shows that
\begin{equation*}
\sum_{|a|\le2N-24}|\p\tilde Z^au^I(t,x)|\ls(\ve+\ve_1^2)\w{x}^{-1/2}\w{t-|x|}^{-1}(1+t)^{11\ve_2+1/2}.
\end{equation*}
This, together with \eqref{loc:improv}, yields \eqref{BA1:improv1}.

Next, we prove \eqref{BA1:improv} by use of \eqref{BA1:improv1}.
By applying \eqref{InLW:dpw} to $\tilde Z^au^I$ again for $|a|\le2N-33$,
we start to improve the estimate \eqref{BA1:improv4}.

In the region $|y|\le1+s/2$, from \eqref{eqn:high}, \eqref{BA1}, \eqref{BA3:improv} and \eqref{BA1:improv1}, one has
\begin{equation}\label{BA1:improv5}
\begin{split}
\sum_{|a|\le2N-24}|\Box Z^a u^I(s,y)|&\ls\sum_{\substack{|b|+|c|+|d|\le2N-24}}|Z^bu||\p Z^cu||\p Z^du|\\
&\ls\ve_1^3\w{s+|y|}^{7\ve_2-1/2}(1+s)^{11\ve_2+1/2}\w{y}^{-1}\w{s-|y|}^{-2-\ve_2}.
\end{split}
\end{equation}

In the region $|y|\ge1+s/2$, it can be concluded from \eqref{eqn:high}, \eqref{null:sturcture}, \eqref{BA3:improv}, \eqref{BA2:improv} and \eqref{BA1:improv1} that
\begin{equation}\label{BA1:improv6}
\begin{split}
\sum_{|a|\le2N-24}|\Box Z^au^I(s,y)|&\ls\sum_{\substack{|b|+|c|+|d|\le2N-24}}|Z^bu||\bar\p Z^cu||\p Z^du|\\
&\ls\ve_1^3\w{s+|y|}^{17\ve_2-2}(1+s)^{11\ve_2+1/2}\w{y}^{-1/2}\w{s-|y|}^{-1-\ve_2}.
\end{split}
\end{equation}
Collecting \eqref{BA1:improv5} and \eqref{BA1:improv6} yields
\begin{equation*}
\begin{split}
\sum_{|a|\le2N-33}\sum_{|b|\le9}\sup_{(s,y)\in[0,t]\times(\overline{\R^2\setminus\cK_2})}\w{y}^{1/2}\cW_{3/2
+\ve_2,1+\ve_2}(s,y)|Z^b\Box Z^au^I(s,y)|\ls\ve_1^3(1+t)^{29\ve_2}.
\end{split}
\end{equation*}
This, together with \eqref{loc:improv} and \eqref{BA1:improv3}, implies \eqref{BA1:improv}.
\end{proof}

\begin{lemma}\label{lem:4-5}
Under the assumptions of Theorem \ref{thm1}, let $u$ be the solution of \eqref{NLW} and suppose that \eqref{BA1}-\eqref{BA3} hold.
Then we have
\begin{equation}\label{loc:impr:A}
\sum_{|a|\le2N-27}\|\p^au\|_{L^2(\cK_R)}\ls(\ve+\ve_1^2)(1+t)^{-1}.
\end{equation}
\end{lemma}
\begin{proof}
Applying \eqref{pw:crit:dt'} to $\p_t^ju^I$ together with $j\le2N-27$, $\mu=\nu=\ve_2$, \eqref{initial:data},
\eqref{eqn:high} and \eqref{loc:improv}
leads to that for $R_1>1$,
\begin{equation}\label{loc:impr:A1}
\begin{split}
&\quad\w{t}\sum_{j\le2N-27}\|\p_t\p_t^ju^I\|_{L^\infty(\cK_{R_1})}\\
&\ls\ve+\ve_1^2+\sum_{|a|\le2N-24}\sup_{(s,y)\in[0,t]\times\cK}\w{y}^{1/2}\cW_{1+2\ve_2,1}(s,y)|\Box Z^au(s,y)|.
\end{split}
\end{equation}
Substituting \eqref{BA1:improv5} and \eqref{BA1:improv6} into \eqref{loc:impr:A1} yields
\begin{equation}\label{loc:impr:A2}
\sum_{j\le2N-27}\|\p_t\p_t^ju\|_{L^\infty(\cK_{R_1})}\ls(\ve+\ve_1^2)(1+t)^{-1}.
\end{equation}
Therefore, \eqref{loc:impr:A} can be achieved by the same method as in Section \ref{sect:LocEnergy} together with \eqref{u:loc:pw} and \eqref{loc:impr:A2}.
\end{proof}

\begin{lemma}[Improvement of \eqref{BA1}]
Under the assumptions of Theorem \ref{thm1}, let $u$ be the solution of \eqref{NLW} and suppose that \eqref{BA1}-\eqref{BA3} hold.
Then
\begin{equation}\label{BA1:impr}
\sum_{|a|\le2N-42}|\p Z^au|\ls(\ve+\ve_1^2)\w{x}^{-1/2}\w{t-|x|}^{-1}.
\end{equation}
\end{lemma}
\begin{proof}
Utilizing \eqref{InLW:dpw} to $\tilde Z^au^I$ for $|a|\le2N-42$ and $\mu=\nu=\ve_2$ yields
\begin{equation}\label{BA1:impr1}
\begin{split}
&\quad\w{x}^{1/2}\w{t-|x|}|\p\tilde Z^au^I|\\
&\ls\ve+\sum_{|b|\le8}\sup_{s\in[0,t]}\w{s}(\|\p^b\Box(\tilde Z^a-Z^a)u^I(s,\cdot)\|_{L^2(\cK_4)}
+\|\p^b\Box Z^au^I(s,\cdot)\|_{L^2(\cK_4)})\\
&+\sum_{|b|\le9}\sup_{(s,y)\in[0,t]\times(\overline{\R^2\setminus\cK_2})}\w{y}^{1/2}\cW_{3/2
+\ve_2,1+\ve_2}(s,y)|Z^b\Box Z^au^I(s,y)|,
\end{split}
\end{equation}
where we have used \eqref{initial:data}.
By \eqref{loc:impr:A}, one obtains that for $|a|\le2N-42$,
\begin{equation}\label{BA1:impr2}
\begin{split}
\sum_{|b|\le8}\sup_{s\in[0,t]}\w{s}(\|\p^b\Box(\tilde Z^a-Z^a)u^I(s,\cdot)\|_{L^2(\cK_4)}
+\|\p^b\Box Z^au^I(s,\cdot)\|_{L^2(\cK_4)})\ls\ve+\ve_1^2.
\end{split}
\end{equation}
With \eqref{BA1:improv} instead of \eqref{BA1:improv1}, \eqref{BA1:improv5} and \eqref{BA1:improv6} can be improved to
that in the region $|y|\le1+s/2$
\begin{equation}\label{BA1:impr3}
\begin{split}
\sum_{|a|\le2N-33}|\Box Z^a u^I(s,y)|&\ls\sum_{\substack{|b|+|c|+|d|\le2N-33}}|Z^bu||\p Z^cu||\p Z^du|\\
&\ls\ve_1^3\w{s+|y|}^{7\ve_2-1/2}(1+s)^{29\ve_2}\w{y}^{-1}\w{s-|y|}^{-2-\ve_2};
\end{split}
\end{equation}
in the region $|y|\ge1+s/2$,
\begin{equation}\label{BA1:impr4}
\begin{split}
\sum_{|a|\le2N-33}|\Box Z^au^I(s,y)|&\ls\sum_{\substack{|b|+|c|+|d|\le2N-33}}|Z^bu||\bar\p Z^cu||\p Z^du|\\
&\ls\ve_1^3\w{s+|y|}^{17\ve_2-2}(1+s)^{29\ve_2}\w{y}^{-1/2}\w{s-|y|}^{-1-\ve_2}.
\end{split}
\end{equation}
Collecting \eqref{BA1:impr1}-\eqref{BA1:impr4} with \eqref{loc:impr:A} leads to \eqref{BA1:impr}.
\end{proof}

\begin{lemma}[Improvement of \eqref{BA3}]
Under the assumptions of Theorem \ref{thm1}, let $u$ be the solution of \eqref{NLW} and suppose that \eqref{BA1}-\eqref{BA3} hold.
Then
\begin{equation}\label{BA3:impr}
\sum_{|a|\le2N-29}|Z^au|\ls(\ve+\ve_1^2)\w{t+|x|}^{\ve_2-1/2}.
\end{equation}
\end{lemma}
\begin{proof}
Utilizing \eqref{InLW:pw} to $\tilde Z^au^I$ for $|a|\le2N-29$, $\mu=\nu=\ve_2$ with \eqref{initial:data} and \eqref{loc:improv} yields
\begin{equation*}
\begin{split}
&\frac{\w{t+|x|}^{1/2}\w{t-|x|}^{\ve_2}}{\ln^2(2+t+|x|)}|\tilde Z^au^I(t,x)|\\
&\ls\ve+\ve_1^2+\sum_{|b|\le5}\sup_{(s,y)\in[0,t]\times(\overline{\R^2\setminus\cK_2})}\w{y}^{1/2}\cW_{1+2\ve_2,1}(s,y)|\p^b\Box Z^au^I(s,y)|.
\end{split}
\end{equation*}
This, together with \eqref{loc:improv}, \eqref{BA1:improv5} and \eqref{BA1:improv6}, implies \eqref{BA3:impr}.
\end{proof}

\begin{lemma}[Improvement of \eqref{BA2}]\label{lem:4-8}
Under the assumptions of Theorem \ref{thm1}, let $u$ be the solution of \eqref{NLW} and suppose that \eqref{BA1}-\eqref{BA3} hold.
Then
\begin{equation}\label{BA2:impr}
\sum_{|a|\le2N-42}|\bar\p Z^au|\ls(\ve+\ve_1^2)\w{x}^{-1/2}\w{t+|x|}^{\ve_2-1}.
\end{equation}
\end{lemma}
\begin{proof}
In the region $|x|\le1+t/2$, \eqref{BA1:impr} leads to \eqref{BA2:impr}.
The proof of \eqref{BA2:impr} in the region $|x|\ge1+t/2$ is similar to that of \eqref{BA2:improv} with a
better estimate than \eqref{BA2:improv2}.
Indeed, it holds that
\begin{equation*}
\sum_{|a|\le2N-42}|\Box Z^au^I(s,y)|\ls\ve_1^3\w{s+|y|}^{\ve_2-1/2}\w{y}^{-1}\w{s-|y|}^{-2},
\end{equation*}
where \eqref{BA1:impr} and \eqref{BA3:impr} are used.
Therefore, \eqref{BA2:improv3} can be also improved to
\begin{equation*}
\begin{split}
|\p_+(r^{1/2}Z^au^I)(t,x)|&\ls\ve\w{x}^{-1}+\ve_1^3(r+t)^{\ve_2-1}\int_0^t(1+|r+t-2s|)^{-2}ds\\
&\quad+(\ve+\ve_1^2)(r+t)^{\ve_2-1/2}\int_0^t(r+t-s)^{-3/2}ds\\
&\ls(\ve+\ve_1^2)\w{t+|x|}^{\ve_2-1}.
\end{split}
\end{equation*}
Together with \eqref{BA2:improv1}, this yields \eqref{BA2:impr}.
\end{proof}

\subsection{Proof of Theorem \ref{thm1}}
\begin{proof}[Proof of Theorem \ref{thm1}]
Collecting \eqref{BA1:impr}, \eqref{BA3:impr} and \eqref{BA2:impr}, we know that there is a constant $C_1\ge1$ such that
\begin{align*}
&\sum_{|a|\le2N-42}|\p Z^au|\le C_1(\ve+\ve_1^2)\w{x}^{-1/2}\w{t-|x|}^{-1},\\
&\sum_{|a|\le2N-42}|\bar\p Z^au|\le C_1(\ve+\ve_1^2)\w{x}^{-1/2}\w{t+|x|}^{\ve_2-1},\\
&\sum_{|a|\le2N-29}|Z^au|\le C_1(\ve+\ve_1^2)\w{t+|x|}^{\ve_2-1/2}.
\end{align*}
By choosing $\ve_1=4C_1\ve,\ve_0=\frac{1}{16C_1^2}$, then for $N\ge42$, \eqref{BA1}-\eqref{BA3} can be improved to
\begin{align*}
&\sum_{|a|\le N}|\p Z^au|\le\frac{\ve_1}{2}\w{x}^{-1/2}\w{t-|x|}^{-1},\\
&\sum_{|a|\le N}|\bar\p Z^au|\le\frac{\ve_1}{2}\w{x}^{-1/2}\w{t+|x|}^{\ve_2-1},\\
&\sum_{|a|\le N}|Z^au|\le\frac{\ve_1}{2}\w{t+|x|}^{\ve_2-1/2}.
\end{align*}
Together with the local existence of classical solution to the initial boundary value problem
of the hyperbolic equation, this yields
that problem \eqref{NLW} admits a global solution $u\in\bigcap\limits_{j=0}^{2N+1}C^{j}([0,\infty), H^{2N+1-j}(\cK))$.
Furthermore, \eqref{thm1:decay:a}-\eqref{thm1:decay:c} and \eqref{thm1:decay:LE} can be obtained by \eqref{loc:impr:A}, \eqref{BA1:impr}, \eqref{BA3:impr} and \eqref{BA2:impr}, respectively.
\end{proof}

\vskip 0.2 true cm

{\bf \color{blue}{Conflict of Interest Statement:}}

\vskip 0.2 true cm

{\bf The authors declare that there is no conflict of interest in relation to this article.}

\vskip 0.2 true cm
{\bf \color{blue}{Data availability statement:}}

\vskip 0.2 true cm

{\bf  Data sharing is not applicable to this article as no data sets are generated
during the current study.}

\vskip 0.2 true cm

\end{document}